\documentclass[11pt]{amsart}

\usepackage[english]{babel}
\usepackage[latin1]{inputenc}
\usepackage{amsthm,amsmath,amsfonts, amssymb,bbm}
\usepackage{hyperref}
\usepackage{enumitem,stmaryrd}
\usepackage{graphicx,eepic}

\def \ZZ{\mathbb Z}
\def \RR{\mathbb R}
\def \EE{\mathbb E}
\def \PP{\mathbb P}

\def \SS{\mathbb S}

\def \N{\mathcal N}

\def \sclr#1#2{\langle #1,#2\rangle}
\def \open#1{#1^{o}}
\def \close#1{\overline{#1}}

\renewcommand{\geq}{\geqslant}
\renewcommand{\leq}{\leqslant}

\usepackage{color}
\definecolor{darkgreen}{rgb}{0,0.4,0}
\definecolor{MyDarkBlue}{rgb}{0,0.08,0.50}
\definecolor{BrickRed}{rgb}{0.65,0.08,0}

\hypersetup{
colorlinks=true,       % false: boxed links; true: colored links
    linkcolor=blue,          % color of internal links
    citecolor=red,        % color of links to bibliography
    filecolor=BrickRed,      % color of file links
    urlcolor=darkgreen        % color of external links
}

\theoremstyle{plain}
\newtheorem{theorem}{Theorem}%[section]
\newtheorem{proposition}[theorem]{Proposition}
\newtheorem{corollary}[theorem]{Corollary}
\newtheorem{lemma}[theorem]{Lemma}

\theoremstyle{definition}

\begin{document}

\author[R.~Garbit]{Rodolphe Garbit}
\address{Universit\'e d'Angers\\D\'epartement de Math\'ematiques\\ LAREMA\\ UMR CNRS 6093\\ 2 Boulevard Lavoisier\\49045 Angers Cedex 1\\ France}
\email{rodolphe.garbit@univ-angers.fr}
\author[K.~Raschel]{Kilian Raschel}
\address{CNRS\\ F\'ed\'eration Denis Poisson\\ Laboratoire de Math\'ematiques et Physique Th\'eorique\\UMR CNRS 7350\\Parc de Grandmont\\ 37200 Tours\\ France}
\email{kilian.raschel@lmpt.univ-tours.fr}

\title[On the exit time from a cone for random walks with drift]{On the exit time from a cone for random walks with~drift}

\thanks{}

\date{\today}

\begin{abstract} 
We compute the exponential decay of the probability that a given multi-dimensional random walk stays in a convex cone up to time $n$, as $n$ goes to infinity. We show that the latter equals the minimum, on the dual cone, of the Laplace transform of the random walk increments. As an example, our results find applications in the counting of walks in orthants, a classical domain in enumerative combinatorics.
\end{abstract}

\maketitle

\section{Introduction and main results}
\label{sec:intro}

\subsection{General context}
For general random processes in $\RR^d$, $d\geq 1$ (including in particular Brownian motion and random walks), it is at once important and natural to study the first exit times $\tau_K$ from certain domains $K$. Precisely, for discrete-time random processes $(S_n)_{n\geq0}$, $\tau_K$ is defined by 
\begin{equation}
\label{eq:def_tau_K}
     \tau_K:=\inf\{n\geq 1: S_n\notin K\}.
\end{equation}
Indeed, these random times carry much valuable information on the process. As an example, the fruitful theory of random walks fluctuations (see, e.g., Spitzer \cite{Sp64}) is based on the analysis of the $\tau_K$ for compact domains $K$.

In a recent past (1990 to present), the case of cones $K$ has arisen a great interest in the mathematical community, due to interactions with many areas: First, certain random walks in conical domains can be treated with representation theory \cite{Bi91,Bi92} (in that case, the cones are Weyl chambers related to Lie algebras). Further, the exit times $\tau_K$ are crucial to construct conditioned random walks in cones, which appear in the theory of quantum random walks \cite{Bi91,Bi92}, random matrices \cite{Dy62}, non-colliding random walks \cite{DeWa10,EiKo08}, etc. In another direction, the probability 
\begin{equation}
\label{eq:def_proba}
     \mathbb P^x[\tau_K>n]
\end{equation}
admits a direct combinatorial interpretation in terms of the number of walks starting from $x$ and staying in the cone $K$ up to time $n$. These counting numbers are particularly important in enumerative combinatorics \cite{BoMi10,FaRa12,JoMiYe13}, and are the topic of many recent studies.

For processes with no drift, the exit times $\tau_K$ from cones are now well studied in the literature. The case of Brownian motion was solved by DeBlassie \cite{DB87} (see also Ba{\~n}uelos and Smits \cite{BaSm97}): He showed that the probability \eqref{eq:def_proba} satisfies a certain partial differential equation (the heat equation), and he solved it in terms of hypergeometric functions. Concerning discrete-time random processes, in the one-dimensional case, the asymptotic behavior of the non-exit probability \eqref{eq:def_proba} is well known, as well as that of
\begin{equation}
\label{eq:def_proba_llt}
     \mathbb P^x[S_n = y,\tau_K>n]
\end{equation}
(called a local limit theorem), thanks to the theory of fluctuations of random walks \cite{Sp64}. In higher dimension, some sporadic cases have first been analyzed: We may cite \cite{DoOC05}, for which there exists a strong underlying algebraic structure (certain reflexion groups are finite), or the case of Weyl chambers, which has been considered in \cite{DeWa10,EiKo08}. For more general cones, but essentially for random walks with increments having a finite support, Varopoulos \cite{Va99} gave lower and upper bounds for the probability \eqref{eq:def_proba}. The first author of the present article showed in \cite{Gar07} that for general random walks, the probability \eqref{eq:def_proba} does not decay exponentially fast. More recently, Denisov and Wachtel \cite{DeWa11} provided the exact asymptotics for both \eqref{eq:def_proba} and \eqref{eq:def_proba_llt}. 

For processes with drift, much less is known. Concerning Brownian motion, one of the first significant results is due to Biane, Bougerol and O'Connell \cite{BiBoOC05}, who derived the asymptotics of the non-exit probability \eqref{eq:def_proba} in the case of Weyl chambers of type $A$, when the drift is inside of the cone. Later on, by using different techniques, Pucha\l a and Rolski \cite{PuRo08} obtained (also in the context of Weyl chambers) the asymptotics of \eqref{eq:def_proba} without any hypothesis on the drift. In \cite{GaRa13} we gave, for Brownian motion with a given arbitrary drift, the asymptotics of \eqref{eq:def_proba} for a large class of cones.

As for random walks $(S_n)_{n\geq0}$ with increments having a common distribution $\mu$, the exponential decay of \eqref{eq:def_proba_llt} is known: It equals the global minimum on $\RR^d$ of the Laplace transform of $\mu$:
\begin{equation}
\label{eq:def_Laplace_transform}
     L_{\mu}(x):=\EE_{\mu}[e^{\sclr{x}{S_{n+1}-S_n}}]=\int_{\RR^d}e^{\sclr{x}{y}}\mu(\text{d}y).
\end{equation}
This was first proved by Iglehart \cite{Ig74} for one-dimensional random walks. For more general walks, this was shown and used by many authors (see, e.g., \cite{DeWa11,Gar07}). Regarding now the asymptotic behavior of the probability \eqref{eq:def_proba}, the case $d=1$ is known, see \cite{Do89}.
\begin{quote}
{\it It is the aim of this paper to give, for a very broad class of random walks and cones, in any dimension, the exponential decay of the non-exit probability \eqref{eq:def_proba}. We shall also relate its value to the Laplace transform, by proving that it equals the minimum of this function {\em on the dual cone}; we give the exact statement (Theorem \ref{maintheorem}) in Subsection \ref{subsec:main_results}.}
\end{quote}

Our main motivation comes from the possible applications to lattice path enumeration. Indeed, our results provide the first unified treatment of the question of determining the growth constant for the number of lattice paths confined to the positive orthant. They also solve a conjecture on these numbers stated in \cite{JoMiYe13}. However, we would like to emphasize that our results are much more general (see Section \ref{sec:assumptions}).

%A few days\footnote{The first version of \cite{Du13} was published on arXiv server on June 25, 2013 while the first version of the present paper appeared on HAL (french) server on June 26, 2013.} before completing the first version of the present article, we learnt about the article \cite{Du13}, written simultaneously and independently of ours, in which Duraj obtained in some particular case the exact asymptotics of the non-exit probability \eqref{eq:def_proba} for lattice random walks. We discuss the difference between our result and his in the following section.
Simultaneously and independently of us, Duraj \cite{Du13} obtained in some particular case
%(namely, when the drift turns toward the origin of the cone---see Subsection \ref{subsec:discussion} for a more precise statement)
the exact asymptotics of the non-exit probability \eqref{eq:def_proba} for lattice random walks. In the following section we introduce our main ideas and tools, and we discuss the difference between our results and his.

\subsection{Preliminary discussion}
\label{subsec:discussion}

Let $(S_n)_{n\geq0}=(S^{(1)}_n,\ldots,S^{(d)}_n)_{n\geq0}$ be the canonical random walk on $\RR^d$. Given any probability measure $\mu$ on $\RR^d$ and $x\in\RR^d$, we denote by $\PP^x_{\mu}$ the probability measure under which $(S_n)_{n\geq0}$ is a random walk started at $x$ whose independent increments $(S_{n+1}-S_n)_{n\geq0}$ have common distribution $\mu$.

The standard idea to handle the case of random walks with non-zero drift is to carry out an exponential change of measure. More precisely, if $z$ is a point in $\RR^d$ such that $L_{\mu}(z)$ is finite, then we can consider the new probability measure
\begin{equation*}
     \mu_z(\text{d}y)=\frac{e^{\sclr{z}{y}}}{L_{\mu}(z)}\mu(\text{d}y).
\end{equation*}     
It is {\it theoretically} possible to compare the behavior of the random walk under $\PP_{\mu}$ with its behavior under $\PP_{\mu_z}$ thanks to Cram\'er's formula (see Lemma \ref{lem:cramer}). For example, for the local probabilities, this formula gives
\begin{equation*}
     \PP_{\mu}^x[S_n=y, \tau_K>n]=L_{\mu}(z)^ne^{\sclr{z}{x-y}}\PP_{\mu_z}^x[S_n=y,\tau_K>n].
\end{equation*}
Since the asymptotic behavior of those probabilities are now well known when the random walk has no drift (see \cite{DeWa11}), the general problem can be solved if one can find a point $z$ such that the distribution $\mu_z$ is {\em centered}. It is also well known that this condition is fulfilled if and only if $z$ is a critical point for $L_{\mu}$ (under the assumption that $L_{\mu}$ be finite in a neighborhood of $z$). By convexity of $L_{\mu}$, this means that one has to find a local, hence {\em global minimum point} $z=x_0$ in $\RR^d$.

This approach is used by Duraj in \cite{Du13} to analyze the non-exit probability \eqref{eq:def_proba}. Indeed, for a lattice random walk, one can sum the contribution of each $y$ to eventually obtain (below $\mu_0$ is an abbreviation for $\mu_{x_0}$)
\begin{equation*}
\PP_{\mu}^x[\tau_K>n]=L_{\mu}(x_0)^n\sum_{y\in K\cap\ZZ^d}e^{\sclr{x_0}{x-y}}\PP_{\mu_{0}}^x[S_n=y,\tau_K>n].
\end{equation*}
But then, one needs to impose an additional condition on the position of the global minimum point $x_0$ with respect to $K$ so as to ensure that the infinite sum of asymptotics will be convergent as well. This technical assumption on $x_0$ done in \cite{Du13} happens to have a very natural interpretation in the light of our analysis. Indeed, for the non-exit probability, Cram\'er's formula (applied with any $z$) gives
\begin{equation}
\label{eq:cramer_formula} 
\PP_{\mu}^x[\tau_K>n]=L_{\mu}(z)^ne^{\sclr{z}{x}}\EE_{\mu_0}^x[e^{-\sclr{z}{S_n}},\tau_K>n],
\end{equation}
and one sees that the main difficulty will arise because of the exponential term inside the expectation.

Let $K^*$ denote the {\em dual} cone associated with $K$, that is, the closed convex cone defined by
\begin{equation}
\label{eq:def_dual_cone}
     K^*:=\{z\in\mathbb R^d: \sclr{x}{z}\geq 0,\forall x\in K\},
\end{equation}  
where $\sclr{x}{z}$ denotes the standard inner product. If $z$ belongs to $K^*$, it immediately follows from \eqref{eq:cramer_formula} that
\begin{equation*}
\PP_{\mu}^x[\tau_K>n]\leq L_{\mu}(z)^ne^{\sclr{z}{x}}.
\end{equation*}
Hence, the infimum $\rho$ of the Laplace transform on $K^*$ is always an upper bound of the exponential rate, i.e., 
\begin{equation*}
\limsup_{n\to\infty}\PP_{\mu}^x[\tau_K>n]^{1/n}\leq \rho:=\inf_{K^*}L_{\mu}.
\end{equation*}
Our main result shows that, in fact, when the infimum $\rho$ is a minimum, it is also a lower bound of the above quantity.
Thus $\rho$ is the value of the exponential decreasing rate of the non-exit probability. 
%In the light of equation \eqref{eq:cramer_formula}, this gives the heuristic rule:
%
%\begin{quote} 
%{\it The right location to perform Cram\'er's transformation is the point $x^*\in K^*$ where the Laplace transform reaches its minimum on the dual cone $K^*$.}
%\end{quote}
It is now easily seen that assumptions 1 and 5 in \cite{Du13} on the global minimum $x_0$ imply that it belongs to the interior $\open{(K^*)}$ of the dual cone (and in this case, clearly, $x_0=x^*$). In \cite{Du13} the author then obtains the precise asymptotics of the non-exit probability \eqref{eq:def_proba} in this specific case. %Note also that the bigger the cone, the smaller the dual cone.

The general philosophy of our work is different (and in a sense complementary). We shall only focus on the exponential rate
\begin{equation*}
     \rho_x:=\limsup_{n\to\infty}\PP_{\mu}^x[\tau_K>n]^{1/n},
\end{equation*}     
and we answer completely the question of determining its value under fairly broad assumptions, regardless the position of the global minimum point $x_0$.

\subsection{Cones and random walks considered}
\label{sec:assumptions}
In this work, we consider a {\em closed convex cone} $K$ with {\em non-empty interior}. Recall that we denote by $K^*$ its {\em dual} cone, which turns out to be particularly relevant for our problem. It is the closed convex cone defined in \eqref{eq:def_dual_cone}. We also set 
\begin{equation*}
     K_{\delta}:=K+\delta v,
\end{equation*}
where $\delta\in \RR$ and $v$ is some fixed vector in $\open{K}$, the interior of $K$.

%Let $(S_n)_{n\geq0}=(S^{(1)}_n,\ldots,S^{(d)}_n)_{n\geq0}$ be the canonical random walk on $\RR^d$. Given any probability measure $\mu$ on $\RR^d$ and $x\in\RR^d$, we note $\PP^x_{\mu}$ the probability measure under which $(S_n)_{n\geq0}$ is a random walk started at $x$ whose independent increments $(S_{n+1}-S_n)_{n\geq0}$ have common distribution $\mu$. 

Throughout this paper, we shall make the assumption that $\mu$ is truly $d$-dimensional in the following sense:
\begin{enumerate}[label=(H\arabic{*}),ref={\rm (H\arabic{*})}]
\setcounter{enumi}{0}
\item\label{H1}The support of the probability measure $\mu$ is not included in any linear hyperplane.
\end{enumerate}
For a square-integrable probability measure $\mu$ with mean $m$ and variance-covariance matrix $\Gamma$, it is well known that the minimal (with respect to inclusion) affine subspace $A$ such that $\mu(A)=1$ is $m+(\ker\Gamma)^{\perp}$. Hence, the condition in~\ref{H1} holds if and only if $m+(\ker\Gamma)^\perp$ is not included in any hyperplane (or equivalently, if and only if $\Gamma$ is non-degenerate or $\dim(\ker\Gamma)=1$ and $m\notin (\ker\Gamma)^{\perp}$). Notice that in the case where $m=0$, the assumption \ref{H1} is equivalent to $\ker\Gamma=\{0\}$, i.e., $\Gamma$ is non-degenerate.

As pointed out in the preceding section, our analysis of the decreasing rate of the non-exit probability requires the existence of a minimum point for the Laplace transform $L_{\mu}$ on the dual cone. Thus we shall impose the following technical condition:
\begin{enumerate}[label=(H\arabic{*}),ref={\rm (H\arabic{*})}]
\setcounter{enumi}{1}
\item\label{H2} There exists a point $x^*\in K^*$ and an open neighborhood $V$ of $x^*$ in $\mathbb R^d$ such that $L_{\mu}(x)$ is finite for all $x\in V$, and $x^*$ is a minimum point of $L_{\mu}$ restricted to $K^*\cap V$.
\end{enumerate}

It is worth noting that we do not assume the existence of moments of $\mu$. Hypothesis \ref{H2} implies the existence of these moments only in the case where $x^*=0$.

In view of applications, we will prove in Subsection \ref{subsec:geometric_interpretation}  that for random walks with all exponential moments (i.e., $L_{\mu}(x)$ is finite for all $x\in \RR^d$), the
condition \ref{H2} is equivalent to the more geometric-flavoured condition: 
\begin{enumerate}[label=(H\arabic{*}'),ref={\rm (H\arabic{*}')}]
\setcounter{enumi}{1}
\item\label{H2bis} The support of $\mu$ is not included in any half-space $u^{-}:=\{x\in\RR^d:\sclr{u}{x}\leq 0\}$ with $u\in K^{*}\setminus\{0\}$.
\end{enumerate}

\subsection{Main results}
\label{subsec:main_results}
We are now in position to state our main result:

\begin{theorem}
\label{maintheorem}
Suppose $\mu$ satisfies \ref{H1} and \ref{H2}. Then,
\begin{equation*}
     \lim_{n\to\infty}\PP^x_{\mu}[\tau_K>n]^{1/n}=L_{\mu}(x^*),
\end{equation*}     
for all $x\in K_{\delta}$, for some constant $\delta\geq 0$.
\end{theorem}

For a large class of random walks and cones, Theorem \ref{maintheorem} gives the universal recipe to compute the exponential decay of the non-exit probability. Notice that the latter is independent of the starting point $x$. This is not the case when \ref{H2} is not satisfied and we shall illustrate this phenomenon in Section \ref{sec:walks_support_half-space}, with the walks in the quarter-plane having transition probabilities as in Figure \ref{fig:degenerated_walks}.

Let us point out that, in general, there is no explicit link between the position of the drift $m$ of the random walk (if it exists), the position of $x^*$ and the value $L_{\mu}(x^*)$ of the decreasing rate, except in the case where $m$ belongs to the cone $K$. As shown in the next lemma, the fact that $m\in K$ is a necessary and sufficient condition for having $L_{\mu}(x^*)=1$ (i.e., a non-exponential decay of the non-exit probability).

\begin{lemma}
\label{Lxet=1} 
Assume \ref{H1} and \ref{H2}. Then $L_{\mu}(x^*)=1$ if and only if $x^*=0$. In addition,
if the drift $m=\int_{\mathbb R^d} y \mu(\textnormal{d}y)$ exists (i.e., if $\mu$ admits a moment of order $1$), then $m$ belongs to $K$ if and only if $x^*=0$.
\end{lemma}

Theorem \ref{maintheorem} in itself does not provide any explicit value for the constant $\delta$, but such a value can be found {a posteriori} thanks to the following: 

\begin{proposition}
\label{securitycone}
The statement in Theorem~\ref{maintheorem} holds for any $\delta\geq 0$ for which there exists $n_0\geq 1$ such that
\begin{equation*}
     \PP^0_{\mu}[\tau_{K_{-\delta}}>n_0, S_{n_0}\in \open{K}]>0.
\end{equation*}     
\end{proposition}

Proofs of Theorem~\ref{maintheorem},  Lemma \ref{Lxet=1} and Proposition~\ref{securitycone} are postponed to Subsection \ref{subsec:proof_proof}.

In order to illustrate Theorem \ref{maintheorem}, it is interesting to compare its content with the corresponding result known for Brownian motion with drift, in the light of the recent paper \cite{GaRa13}. It is proved there that, for Brownian motion $(B_t)_{t\geq0}$ with drift $a\in\mathbb R^d$, the non-exit probability admits the asymptotics (in the continuous case, the exit time from $K$ is defined by $\tau_K:=\inf\{t>0 : B_t\notin K\}$)
\begin{equation}
      \PP^x[\tau_K>t]=\kappa h(x) t^{-\alpha} e^{-\gamma t}(1+o(1)),\quad t\to\infty,
\end{equation}
where  $\gamma:=d(a,K)^2/2$. Therefore 
\begin{equation*} 
     \lim_{t\to\infty}\PP^x[\tau_K>t]^{1/t}=e^{-d(a,K)^2/2}.
\end{equation*}
Let us compare with the value given by Theorem \ref{maintheorem} for the random walk $(B_n)_{n\geq 0}$. Its distribution $\mu$ is Gaussian with mean $a$ and identity variance-covariance matrix; therefore
\begin{equation*}
     L_{\mu}(x)= e^{\vert x\vert^2/2+\sclr{x}{a}}.
\end{equation*}
The minimum on $K^*$ of $\vert x\vert^2/2+\sclr{x}{a}$ is obviously the minimum on the {\em polar} cone $K^\sharp=-K^*$ of
\begin{equation*}
\vert x\vert^2/2-\sclr{x}{a}=\vert x-a\vert^2/2-\vert a\vert^2/2.
\end{equation*}
It is reached at $x = p_{K^\sharp}^\perp (a)$, the orthogonal projection of $a$ on $K^\sharp$, and an easy computation shows that the minimum value is
\begin{equation*}
   \vert p_{K^\sharp}^\perp (a)-a\vert^2/2-\vert a\vert^2/2 = -\vert a-p_K^\perp(a)\vert^2/2=-d(a,K)^2/2,
\end{equation*}
where we have used Moreau's decomposition theorem which asserts that, for any convex cone $K$,  $a$ is the orthogonal sum of $p_{K}^\perp (a)$ and $p_{K^\sharp}^\perp (a)$. We thus have 
\begin{equation*}
     \min_{K^*}L_{\mu}=e^{-d(a,K)^2/2},
\end{equation*}
which means that the exponential decreasing rate is the same for Brownian motion $(B_t)_{t\geq 0}$ and for the ``sampled'' Brownian motion $(B_n)_{n\geq 0}$, as one could expect.

  \unitlength=0.6cm
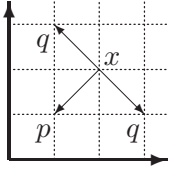
\begin{figure}[t]
  \begin{center}
\begin{tabular}{ccccc}
    \hspace{-0.9cm}
        \begin{picture}(4,4.5)
    \thicklines
    \put(1,1){{\vector(1,0){3.5}}}
    \put(1,1){\vector(0,1){3.5}}
    \thinlines
    \put(3,3){\vector(1,-1){1}}
    \put(3,3){\vector(-1,1){1}}
    \put(3,3){\vector(-1,-1){1}}
    \linethickness{0.1mm}
    \put(1,2){\dottedline{0.1}(0,0)(3.5,0)}
    \put(1,3){\dottedline{0.1}(0,0)(3.5,0)}
    \put(1,4){\dottedline{0.1}(0,0)(3.5,0)}
    \put(2,1){\dottedline{0.1}(0,0)(0,3.5)}
    \put(3,1){\dottedline{0.1}(0,0)(0,3.5)}
    \put(4,1){\dottedline{0.1}(0,0)(0,3.5)}
    \put(1.6,3.5){$q$}
    \put(1.6,1.5){$p$}
    \put(3.6,1.5){$q$}
    \put(3.1,3.1){$x$}
    \end{picture}
    \end{tabular}
  \end{center}
  \vspace{-4mm}
\caption{Random walks considered in Section \ref{sec:walks_support_half-space} ($p+2q=1$, $p,q>0$), for different starting points $x$}
\label{fig:degenerated_walks}
\end{figure}

\subsection{Plan of the paper}

The rest of our article is organized as follows: In Section~\ref{sec:proof_main_theorem} we prove Theorem~\ref{maintheorem}. In Section \ref{sec:enumeration} we present an important consequence of Theorem \ref{maintheorem} in the counting of walks in orthants (a topical domain in enumerative combinatorics), see Corollaries \ref{cor:enumeration} and \ref{cor:enumeration2}. In Section \ref{sec:walks_support_half-space} we consider the walks of Figure \ref{fig:degenerated_walks}, for which we prove that contrary to the walks satisfying hypothesis \ref{H2}, the exponential decay depends on the starting point $x$. Finally, in Section \ref{sec:appendix} we prove the non-exponential decay of the non-exit probability for random walks with drift in the cone (a refinement of a theorem of \cite{Gar07}), which is needed for proving our main result.

\section{Proof of the main results}
\label{sec:proof_main_theorem}

This section is organized as follows: In Subsection \ref{subsec:Cramer}, we review some elementary properties of the Laplace transform and present Cram\'er's formula. Then, we prove Theorem~\ref{maintheorem}, Lemma \ref{Lxet=1} and Proposition~\ref{securitycone} in Subsection \ref{subsec:proof_proof}. In Subsection \ref{subsec:geometric_interpretation} we provide a geometric interpretation of our main assumption on the random walk distribution.

\subsection{Cram\'er's formula}
\label{subsec:Cramer}

The Laplace transform of~a~probability distribution $\mu$ is the function $L_{\mu}$ defined for $x\in\RR^d$ by
\begin{equation*}
     L_{\mu}(x):=\int_{\RR^d}e^{\sclr{x}{y}}\mu(\text{d}y).
\end{equation*}
It is clearly a convex function. If $L_{\mu}$ is finite in a neighborhood of the origin, say $\close{B(0,r)}$, then it is well known that $L_{\mu}$ is (infinitely) differentiable in $B(0,r)$, and that its partial derivatives are given by
\begin{equation*}
     \frac{\partial L_{\mu}(x)}{\partial x_i}=\int_{\RR^d}y_i e^{\sclr{x}{y}}\mu(\text{d}y),\quad \forall i\in\llbracket1,d\rrbracket.
\end{equation*}
Therefore, the expectation of $\mu$ is equal to the gradient of $L_{\mu}$ at the origin: $\EE[\mu]=\nabla L_{\mu}(0)$. Notice that $\mu$ is centered if and only if $0$ is a critical point of $L_{\mu}$.

Suppose now that $z$ is a point where $L_{\mu}$ is finite, and let $\mu_z$ denote the probability measure defined by
\begin{equation}
\label{eq:def_mu_zero}
     \mu_z(\text{d}y):=\frac{e^{\sclr{z}{y}}}{L_{\mu}(z)}\mu(\text{d}y).
\end{equation}
The Laplace transform of $\mu_z$ is related to that of $\mu$ by the formula
\begin{equation*}
L_{\mu_{z}}(x)=\frac{L_{\mu}(z+x)}{L_{\mu}(z)}.
\end{equation*}
If in addition $L_{\mu}$ is finite in some ball $\close{B(z,r)}$, then $L_{\mu_{z}}$ is finite in $\close{B(0,r)}$. By consequence, $L_{\mu_{z}}$ is differentiable in $B(0,r)$ and
\begin{equation*}
     \EE[\mu_z]=\nabla L_{\mu_z}(0)=\frac{\nabla L_{\mu}(z)}{L_{\mu}(z)}.
\end{equation*}

The distribution of the random walk under $\PP_{\mu_z}$ is linked to the initial distribution by the following:
\begin{lemma}[Cram\'er's formula]
\label{lem:cramer}
For any measurable and positive function $F:\RR^n\to\RR$, we have
\begin{equation*}
\EE_{\mu}^x[F(S_1,\ldots,S_n)]=L_{\mu}(z)^ne^{\sclr{z}{x}}\EE_{\mu_z}^x[e^{-\sclr{z}{S_n}}F(S_1,\ldots,S_n)].
\end{equation*}
\end{lemma}
\begin{proof} It follows directly from the definition~\eqref{eq:def_mu_zero} of $\mu_z$ that
\begin{equation*}
\mu^{\otimes n}(\text{d}y_1,\ldots,\text{d}y_n)=L_{\mu}(z)^n e^{-\sclr{z}{\sum_{i=1}^ny_i}}\mu_z^{\otimes n}(\text{d}y_1,\ldots,\text{d}y_n).
\end{equation*}
The conclusion is then straightforward.
\end{proof}
Applied to the function $F(s_1,\ldots,s_n)=\Pi_{i=1}^n \mathbbm{1}_{K}(s_i)$, Cram\'er's formula reads
\begin{equation*} 
\PP_{\mu}^x[\tau_K>n]=L_{\mu}(z)^ne^{\sclr{z}{x}}\EE_{\mu_z}^x[e^{-\sclr{z}{S_n}},\tau_K>n],
\end{equation*}
and this implies that for all $x\in\RR^d$,
\begin{equation}
\label{eq:upper_bound}
\limsup_{n\to\infty}\PP_{\mu}^x[\tau_K>n]^{1/n}\leq \inf_{K^*}L_{\mu},
\end{equation}
as already observed at the end of Subsection \ref{subsec:discussion}.

\subsection{Proofs of Theorem~\ref{maintheorem}, Lemma \ref{Lxet=1} and Proposition~\ref{securitycone}}
\label{subsec:proof_proof}

In order to obtain the lower bound for the non-exit probability, we shall use Cram\'er's formula at $z=x^*$. 
The following lemma gives some useful information on the position of the drift of the random walk under the measure changed at $x^*$.

\begin{lemma}
\label{driftattheminimum}
Suppose $\mu$ satisfies \ref{H1} and \ref{H2}. Then, the gradient $\nabla L_{\mu}(x^*)$ belongs to $K$ and is orthogonal to $x^*$.
\end{lemma}

\begin{proof} We first notice that under \ref{H2}, the Laplace transform is finite, hence differentiable, in some neighborhood of $x^*$. It is well known that the equality $(K^*)^*=K$ holds for any closed convex cone, see \cite[Theorem 14.1]{Ro70}. Hence $\nabla L_{\mu}(x^*)$ belongs to $K$ if and only if
\begin{equation*}
     \sclr{\nabla L_{\mu}(x^*)}{y}\geq 0,\quad\forall y\in K^*.
\end{equation*}
So, let $y\in K^*$. Since $K^*$ is a convex cone and $x^*\in K^*$, $x^*+ty$ also belongs to $K^*$ for all $t\geq 0$. Hence, 
thanks to \ref{H2}, the function
\begin{equation*}
     t\in[0,\infty) \mapsto f_y(t):=L_{\mu}(x^*+ty)
\end{equation*}
is differentiable in some neighborhood of $t=0$ and reaches its minimum at $t=0$. This implies
\begin{equation*}
     \sclr{\nabla L_{\mu}(x^*)}{y}=f'_y(0)\geq 0.
\end{equation*}     
Now, if we take $y=x^*$, we have a stronger result since $x^*+tx^*=(1+t)x^*$ belongs to $K^*$ for all $t\geq -1$: the function $f_{x^*}$ is differentiable in some open neighborhood of $t=0$ and has a local minimum point on $[-1,\infty)$ at $t=0$, hence
\begin{equation*}
     \sclr{\nabla L_{\mu}(x^*)}{x^*}=f'_{x^*}(0)=0.
\end{equation*}     
The proof is completed.
\end{proof}

We are now in position to conclude the proof of Theorem \ref{maintheorem}.

\begin{proof}[Proof of Theorem~\ref{maintheorem}]
As already observed, $L_{\mu}(x^*)$ is an upper bound for the exponential decreasing rate
\begin{equation*}
     \limsup_{n\to\infty}\PP^x_{\mu}[\tau_K>n]^{1/n},
\end{equation*}     
and it remains to prove that it is also the right lower bound.
By performing the Cram\'er transformation at $x^*$, we get
\begin{align*}
\label{cramerattheminimum}
     \PP^x_{\mu}[\tau_K>n]&=\rho^n e^{\sclr{x^*}{x}}\EE^{x}_{\mu_*}[e^{-\sclr{x^*}{S_n}}, \tau_K>n]\\
		&\geq \rho^n e^{\sclr{x^*}{x}}e^{-\alpha\sqrt{n}}\PP^{x}_{\mu_*}[\vert\sclr{x^*}{S_n}\vert\leq\alpha\sqrt{n}, \tau_K>n],
\end{align*}
where $\rho=L_{\mu}(x^*)$, $\mu_*(\text{d}y)=\rho^{-1}e^{\sclr{x^*}{y}}\mu(\text{d}y)$, and $\alpha$ is any positive number.
Notice that $\mu_*$ is truly $d$-dimensional, because $\mu$ has this property and both measures have the same support.
Note also that by assumption \ref{H2}, $L_{\mu_*}(x)=\rho^{-1}L_{\mu}(x^*+x)$ is finite in some neighborhood of $x=0$, and therefore $\mu_*$ has all moments.
Since the new drift
\begin{equation*}
     \EE[\mu_*]=m_*=\rho^{-1} \nabla L_{\mu}(x^*)
\end{equation*}
belongs to $K$ and is orthogonal to $x^*$ (by Lemma \ref{driftattheminimum}), it follows from Proposition~\ref{theorem_cras_extended} in Section~\ref{sec:appendix} that there exist $\alpha>0$ and $\delta\geq 0$ such that
\begin{equation*}
     \lim_{n\to\infty}\PP^{x}_{\mu_*}[\vert\sclr{x^*}{S_n}\vert\leq\alpha\sqrt{n}, \tau_K>n]^{1/n}=1,\quad \forall x\in K_{\delta}.
\end{equation*}
Hence, we reach the conclusion that
\begin{equation*}
     \liminf_{n\to\infty}\PP^x_{\mu}[\tau_K>n]^{1/n}\geq \rho
\end{equation*}
for all $x\in K_{\delta}$, and the theorem is proved.
\end{proof}

%\begin{remark}
%If the drift of the original random walk belongs to the cone $K$, the minimum of the Laplace transform on $K^*$ is at the origin of the cone: 
%\begin{equation*}
%     \lim_{n\to\infty}\PP^x_{\mu}[\tau_K>n]^{1/n}=1=L_{\mu}(0)=\min_{K^*} L_{\mu}.
%\end{equation*}  
%In other cases, when the drift is not inside the cone $K$, there is no direct link between the location of the minimum $x^*$ and the location of the drift. \textcolor{red}{Rodolphe}
%\end{remark}

We now give the proof of Lemma \ref{Lxet=1}, which provides a necessary and sufficient condition for having $L_{\mu}(x^* )=1$. 

\begin{proof}[Proof of Lemma \ref{Lxet=1}]
Assume that $L_{\mu}(x^* )=1$ and at the same time $x^*\not=0$. It is well known that $L_{\mu}$ is then finite on $[0, x^*]$, thus strictly convex on that segment (see Subsection \ref{subsec:geometric_interpretation}). Since $L_{\mu}(0)=L_{\mu}(x^*)=1$ it follows that $L_{\mu}(x)<1$ for all $x\in (0, x^*)$. But this open interval is a subset of $K^*$, hence this contradicts the hypothesis \ref{H2} asserting that $x^*$ is a local minimum point on $K^*$. Conversely, that $x^*=0$ implies $L_{\mu}(x^*)=1$ is trivial.

We now turn to the second part of the lemma. First, we know from Lemma \ref{driftattheminimum} that $\nabla L_{\mu}(x^*)$ belongs to $K$. So, if $x^*=0$, then $m=\nabla L_{\mu}(0)$ belongs to $K$ (here we do not need to assume the existence of $m$: it exists because \ref{H2} at $x^*=0$ ensures that $L_{\mu}$ is infinitely differentiable in some neighborhood of $0$, and therefore $\mu$ has all moments). Conversely, assume that $m$ exists and belongs to $K$ and suppose that $x^*\not=0$. Consider the function $g(t)=L_{\mu}(tx^*)$, which is finite on $[0,1]$. Under the assumption $\int \vert y\vert\mu(\text{d}y)<\infty$, it follows by standard arguments that $g(t)$ has a right derivative at $t=0$ given by $g'(0+)=\sclr{x^*}{m}$. Since $m$ belongs to $K$ and $x^*$ to  $K^*$, this derivative is non-negative. So, $g(t)$ must be increasing since it is strictly convex. Thus $L_{\mu}(tx^*)=g(t)<g(1)=L_{\mu}(x^*)$ for all $t\in[0,1)$, and $x ^*$ cannot be a local minimum.
\end{proof}

To conclude this section, we explain how to find a $\delta\geq 0$ for which the statement of Theorem~\ref{maintheorem} holds.

\begin{proof}[Proof of Proposition~\ref{securitycone}]
Recall that $v\in\open{K}$ is fixed and $K_{\delta}=K+\delta v$. We assume that there exist $\delta\geq 0$ and $k\geq 1$ such that
\begin{equation*}
     \PP^0_{\mu}[\tau_{K_{-\delta}}>k, S_{k}\in \open{K}]>0.
\end{equation*}
Therefore, we can find $\epsilon>0$ such that 
\begin{equation*}
     \PP^0_{\mu}[\tau_{K_{-\delta}}>k, S_{k}\in K_{\epsilon}]=\gamma>0,
\end{equation*}
and since $K$ is a convex cone, it satisfies the relation $K+K\subset K$, thus
\begin{equation*}
     \PP^x_{\mu}[\tau_{K}>k, S_{k}-x\in K_{\epsilon}]\geq \gamma,
\end{equation*}
for all $x\in K_{\delta}$ (by inclusion of events). From this, we shall deduce by induction that
\begin{equation}\label{pousselamarche}
p_\ell:=\PP_{\mu}^x[\tau_K>\ell k,S_{\ell k}-x\in K_{\ell\epsilon}]\geq \gamma^\ell,
\end{equation}
for all $\ell\geq 1$ and $x\in K_{\delta}$. Indeed, by the Markov property of the random walk,
\begin{align*}
p_{\ell+1} & \geq \EE_{\mu}^x[\tau_K>\ell k,S_{\ell k}-x\in K_{\ell\epsilon}, \PP_{\mu}^{S_{\ell k}}[\tau_K>k,S_{k}-x\in K_{(\ell+1)\epsilon}]]\\
& \geq p_\ell\cdot\inf_{y} \PP_{\mu}^{y}[\tau_K>k,S_{k}-x\in K_{(\ell+1)\epsilon}],
\end{align*}
where the infimum is taken over all $y\in K$ such that $y-x\in K_{\ell\epsilon}$. Noting that $y-x\in K_{\ell\epsilon}$ and $S_k-y\in K_{\epsilon}$ imply
$S_k-x\in K_{(\ell+1)\epsilon}$, we obtain
\begin{equation*}
      p_{\ell+1}\geq p_\ell\cdot\inf_{y} \PP_{\mu}^{y}[\tau_K>k,S_{k}-y\in K_{\epsilon}].
\end{equation*}
But $x\in K_{\delta}$ and $y-x\in K_{\ell\epsilon}\subset K$ imply $y\in K_{\delta}$. Hence $p_{\ell+1}\geq p_\ell\cdot \gamma$ and~\eqref{pousselamarche} is proved.

Now Theorem~\ref{maintheorem} asserts the existence of some $\delta_0\geq 0$ such that
\begin{equation*}
     \lim_{n\to\infty}\PP^y_{\mu}[\tau_K>n]^{1/n}=L_{\mu}(x^*),
\end{equation*}
for all $y\in K_{\delta_0}$, and we want to prove that the result also holds for $x\in K_{\delta}$. To do this, we shall simply use~\eqref{pousselamarche} in order to push the walk from $K_{\delta}$ to $K_{\delta_0}$. More precisely, choose $\ell\geq 1$ such that $\delta+\ell\epsilon\geq \delta_0$. Then for all $x\in K_{\delta}$ the inclusion $x+K_{\ell\epsilon}\subset K_{\delta_0}$ holds, and thanks to~\eqref{pousselamarche},
\begin{equation*}
     \PP_{\mu}^x[\tau_K>m,S_{m}\in K_{\delta_0}]\geq \gamma^\ell,
\end{equation*}
for $m=k\ell$. By the Markov property, for all $n\geq m$, we have
\begin{align*}
\PP_{\mu}^x[\tau_K>n] & \geq \EE_{\mu}^x[\tau_K>m, S_{m}\in K_{\delta_0},\PP_{\mu}^{S_{m}}[\tau_K>n-m] ]\\
&\geq \gamma^\ell \cdot \inf_{y\in K_{\delta_0}}\PP_{\mu}^{y}[\tau_K>n-m]\\ 
&\geq \gamma^\ell \cdot \PP_{\mu}^{\delta_0v}[\tau_K>n-m], 
\end{align*}
where the last inequality follows by inclusion of events. This implies immediately
\begin{equation*}
     \liminf_{n\to\infty}\PP_{\mu}^x[\tau_K>n]^{1/n}\geq L_{\mu}(x^*).
\end{equation*}
But since the inequality
\begin{equation*}
     \limsup_{n\to\infty}\PP_{\mu}^x[\tau_K>n]^{1/n}\leq L_{\mu}(x^*)
\end{equation*}     
holds for all $x$ (see~\eqref{eq:upper_bound}), Proposition~\ref{securitycone} is proved. 
\end{proof}

\subsection{Geometric interpretation of condition \ref{H2}}
\label{subsec:geometric_interpretation}
The aim of this subsection is to give a geometric interpretation of condition \ref{H2} under some additional condition on the exponential moments. Throughout this section, we fix a closed convex cone $C$ and assume that $\mu$ has all $C$-exponential moments, that is, 
\begin{equation*}
     L_{\mu}(x)<\infty,\quad\forall x\in C.
\end{equation*}    
The strict convexity of the exponential function ensures that
\begin{equation}
\label{eq:eq_not_eq}
     L_{\mu}(ax_1+bx_2)\leq aL_{\mu}(x_1)+bL_{\mu}(x_2),
\end{equation}
for all $x_1\not=x_2\in C$ and $a,b>0$ with $a+b=1$, and that equality occurs if and only if 
\begin{equation*}
     \mu((x_1-x_2)^{\perp})=1,
\end{equation*} 
where $(x_1-x_2)^{\perp}$ denotes the hyperplane orthogonal to $x_1-x_2$. Thus, if~\ref{H1} is satisfied, the equality in \eqref{eq:eq_not_eq} never occurs and $L_{\mu}$ is strictly convex on $C$. 

Let $\SS^{d-1}$ denote the unit sphere of $\RR^d$. Standard arguments involving the convexity of $L_{\mu}$ and the compactness of $C\cap \SS^{d-1}$ show that $L_{\mu}(x)$ goes to infinity as $\vert x\vert\to\infty$ uniformly on $C$ if and only if 
\begin{equation}
\label{eq:coercive_condition}
     \lim_{t\to \infty}L_{\mu}(tu)=\infty,\quad \forall u\in C\cap\SS^{d-1}.
\end{equation}
Hence, the condition~\eqref{eq:coercive_condition} is sufficient for the existence of a global minimum on $C$. Indeed, if it is satisfied, then
there exists $R>0$ such that $L_{\mu}(x)\geq 1$ for all $x\in C$ with $\vert x\vert\geq R$. By continuity, $L_{\mu}$ reaches a minimum on $\close{B(0,R)}\cap C$ which is less than or equal to $1$ (since $L_{\mu}(0)=1$), thus it is a global minimum on $C$. 

The next lemma gives some interesting information on the behavior at infinity of $L_{\mu}$. Recall that $u^-$ denotes the half-space $\{y\in\RR^d : \sclr{u}{y}\leq 0\}$.

\begin{lemma} 
\label{laplace_at_infinity} 
Suppose that $\mu$ has all $C$-exponential moments. For every $u\in C\cap\SS^{d-1}$, the following dichotomy holds:
\begin{enumerate}[label={\rm (\arabic{*})},ref={\rm (\arabic{*})}]
\item If $\mu(u^-)<1$, then 
\begin{equation*}
\lim_{t\to \infty}L_{\mu}(x+tu)=\infty,\quad \forall x\in C.
\end{equation*}
\item If $\mu(u^-)=1$, then
\begin{equation*}
\lim_{t\to \infty}L_{\mu}(x+tu)=\int_{u^{\perp}}e^{\sclr{x}{y}}\mu(\textnormal{d}y),\quad \forall x\in C.
\end{equation*}
\end{enumerate}
\end{lemma}
\begin{proof}
If $\mu(u^-)<1$, then we can find $\epsilon>0$ such that the set $\{y\in\RR^d : \sclr{u}{y}\geq \epsilon\}$ has positive measure,
and the inequality
\begin{equation*}
L_{\mu}(x+tu)\geq \int_{\{y\in\RR^d:\sclr{u}{y}\geq\epsilon\}}e^{\sclr{x}{y}}e^{t\epsilon}\mu(\text{d}y)\geq ce^{t\epsilon},
\end{equation*}
proves the first assertion of Lemma \ref{laplace_at_infinity}. Suppose now on the contrary that $\mu(u^-)=1$. We then may write
\begin{equation*}
     L_{\mu}(x+tu)=\int_{u^{\perp}}e^{\sclr{x}{y}}\mu(\text{d}y)+\int_{\{y\in\RR^d:\sclr{u}{y}<0\}}e^{\sclr{x+tu}{y}}\mu(\text{d}y).
\end{equation*}
The second integral on the right-hand side of the above equation goes to zero as $t$ goes to infinity by the dominated convergence theorem, thus proving the second assertion of Lemma \ref{laplace_at_infinity}.
\end{proof}

\begin{lemma} 
\label{global_minimum_equivalence} Suppose that $\mu$ satisfies \ref{H1} and has all $C$-exponential moments. Then the Laplace transform $L_{\mu}$ has a global minimum on the closed convex cone $C$ if and only if there does not exist any $u\not=0$ in $C$ such that $\mu(u^-)=1$.
\end{lemma}
\begin{proof}
If $\mu(u^-)<1$ for all $u\not=0$ in $C$, then $\lim_{t\to\infty}L_{\mu}(tu)=\infty$ by Lemma~\ref{laplace_at_infinity}, and the function $L_{\mu}$ has a global minimum on $C$ as explained earlier. Now, suppose on the contrary that $\mu(u^-)=1$ for some $u\not=0$ in $C$. Then, by Lemma~\ref{laplace_at_infinity} again, the limit
\begin{equation*}
     h(x):=\lim_{t\to \infty}L_{\mu}(x+tu)
\end{equation*} 
exists and is finite for all $x$. Since any convex cone is a semi-group, $x+tu\in C$ for all $x\in C$ and $t\geq 0$, and consequently 
\begin{equation*}
     h(x)\geq \inf_{C}L_{\mu},\quad \forall x\in C.
\end{equation*} 
But our assumption that $\mu$ satisfies \ref{H1} implies that $L_{\mu}$ is strictly convex, so that $L_{\mu}(x)> h(x)$ for all $x\in C$ (for else the strictly convex function $t\mapsto L_{\mu}(x+tu)$ on $[0,\infty)$ would not have a finite limit). We thus  reach the conclusion that $L_{\mu}(x)>\inf_{C}L_{\mu}$ for all $x\in C$, thereby proving that $L_{\mu}$ has no global minimum on $C$.
\end{proof}

For random walks with all exponential moments, the equivalence between conditions \ref{H2} and \ref{H2bis} under \ref{H1} is now an easy consequence of Lemma~\ref{global_minimum_equivalence}.

\section{Application to lattice path enumeration}
\label{sec:enumeration}

In this section we present an application of our main result (Theorem \ref{maintheorem}) in enumerative combinatorics: Given a finite set $\mathfrak{S}$ of allowed steps, a now classical problem is to study $\mathfrak{S}$-walks in the orthant 
\begin{equation*}
     Q:=(\RR^+)^d=\{x\in\RR^d : x_i\geq 0, \forall i\in \llbracket 1,d\rrbracket\},
\end{equation*}     
that is walks confined to $Q$, starting at a fixed point $x$ (often the origin) and using steps in $\mathfrak{S}$ only. Denote by $f_{\mathfrak{S}}(x,y;n)$ the number of such walks that end at $y$ and use exactly $n$ steps. Many properties of the counting numbers $f_{\mathfrak{S}}(x,y;n)$ have been recently analyzed (the seminal work in this area is \cite{BoMi10}). First, exact properties of them were derived, via the study of their generating function (exact expression and algebraic nature). Such properties are now well established for the case of small steps walks in the quarter-plane, meaning that the step set $\mathfrak{S}$ is included in $\{0,\pm 1\}^2$. More qualitative properties of the $f_{\mathfrak{S}}(x,y;n)$ were also investigated, such as the asymptotic behavior, as $n\to\infty$, of the number of excursions $f_{\mathfrak{S}}(x,y;n)$ for fixed $y$, or that of the total number of walks, 
\begin{equation}
\label{eq:total_walks}
     f_{\mathfrak{S}}(x;n):=\sum_{y\in Q}f_{\mathfrak{S}}(x,y;n).
\end{equation}
Concerning the excursions, several small steps cases have been treated by Bousquet-M\'elou and Mishna \cite{BoMi10} and by Fayolle and Raschel \cite{FaRa12}. Later on, Denisov and Wachtel \cite{DeWa11} obtained the very precise asymptotics of the excursions, for a quite large class of step sets and cones. As for the total number of walks \eqref{eq:total_walks}, only very particular cases are solved, see again \cite{BoMi10,FaRa12}. In a most recent work \cite{JoMiYe13}, Johnson, Mishna and Yeats obtained an upper bound for the exponential growth constant, namely,
\begin{equation*}
     \limsup_{n\to\infty} f_{\mathfrak{S}}(x;n)^{1/n},
\end{equation*}
and proved by comparison with results of \cite{FaRa12} that these bounds are tight for all small steps models in the quarter-plane. In the present article, we find the exponential growth constant of the total number of walks \eqref{eq:total_walks} in any dimension for any model such that:
\begin{enumerate}[label=(H\arabic{*}''),ref={\rm (H\arabic{*}'')}]
\item\label{H1p}The step set $\mathfrak S$ is not included in a linear hyperplane;
\item\label{H2p}The step set $\mathfrak S$ is not included in a half-space $u^-$, with $u\in Q\setminus \{0\}$. 
\end{enumerate}
Our results provide the first unified treatment of this problem of determining the growth constant for the number of lattice paths confined to the positive orthant. In the sequel we shall say that a step set $\mathfrak S$ is proper if it satisfies to \ref{H1p} and \ref{H2p}. Note in particular that the well-known $79$ models of walks in the quarter-plane studied in \cite{BoMi10,FaRa12} (including the so-called $5$ singular walks) satisfy both hypotheses above.

\begin{corollary}
\label{cor:enumeration}
Let $\mathfrak S$ be any proper step set. The Laplace transform of $\mathfrak S$,
\begin{equation*}
     L_\mathfrak S(x) := \sum_{s\in \mathfrak S} e^{\sclr{x}{s}},
\end{equation*}
reaches a global minimum on $Q$ at a unique point $x_0$, and there exists $\delta\geq0$ such that for any starting point $x\in Q_\delta$,
\begin{equation*}
     \lim_{n\to\infty} f_{\mathfrak{S}}(x;n)^{1/n} = L_\mathfrak S(x_0).
\end{equation*}
\end{corollary}

Suppose that in addition:
\begin{enumerate}[label=(H\arabic{*}''),ref={\rm (H\arabic{*}'')}]
\setcounter{enumi}{2}
\item\label{H3p}The step set allows a path staying in $Q$ from the origin to some point in the interior of $Q$.
\end{enumerate}
Then it follows from Proposition \ref{securitycone} that the result in Corollary~\ref{cor:enumeration} holds with $\delta=0$, i.e., it is valid for all $x\in Q$. Note that this assumption is not restrictive from a combinatorial point of view, since if \ref{H3p} is not satisfied, the counting problem is obvious.

\begin{proof}[Proof of Corollary \ref{cor:enumeration}]
Consider a random walk $(S_n)_{n\geq0}$ starting from $x$ such that $\mu$ is the uniform law on $\mathfrak{S}$. Let then $\tau_Q$ denote the first exit time from $Q$. The enumeration problem is related to probabilities in a simple way:
\begin{equation}
\label{eq=counting-probab}
     \mathbb P_{\mu}^x[\tau_Q> n]=\frac{f_\mathfrak S(x;n)}{\vert\mathfrak S \vert^n}.
\end{equation}
Further, it is immediate from our definitions that $L_\mathfrak S(x)=\vert \mathfrak S\vert L_\mu(x)$. Corollary \ref{cor:enumeration} then follows from Theorem \ref{maintheorem} and from the fact that $Q^*=Q$.
\end{proof}

As a consequence, we obtain the following result, which was conjectured in \cite{JoMiYe13}:

\begin{corollary}
\label{cor:enumeration2}
Let $\mathfrak S\subset \mathbb Z^d$ be a proper step set {\rm(}hypotheses \ref{H1p} and \ref{H2p}{\rm)}, which additionally satisfies \ref{H3p}, and let $K_\mathfrak S$ be the growth constant for the total number of walks \eqref{eq:total_walks}. Let $\mathcal P$ be the set of hyperplanes through the origin in $\mathbb R^d$ which do not meet the interior of the first orthant. Given $p\in\mathcal P$, let $K_\mathfrak S( p )$ be the growth constant of the walks on $\mathfrak S$ which are restricted to the side of $p$ which includes the first orthant. Then $K_\mathfrak S=\min_{p\in\mathcal P}K_\mathfrak S( p )$.
\end{corollary}
\begin{proof}
Let us first notice that $\mathcal P$ can be described as the set of hyperplanes $u^{\perp}$ such that $u\in Q\cap\SS^{d-1}$, and that the side of $p=u^{\perp}$ which includes the first orthant is then the half-space $u^+=\{x\in\RR^d : \sclr{x}{u}\geq 0\}$. By Theorem~\ref{maintheorem}, the exponential rate for the random walk associated to the step set $\mathfrak S$ and confined to $u^+$ is the minimum of $L_\mathfrak S/\vert\mathfrak S \vert$ on the dual cone $(u^+)^*=\{tu :t\geq 0\}$. Therefore, the growth constant $K_\mathfrak S( p )$ equals $\min_{t\geq 0}L_\mathfrak S(tu)$, and the equality
\begin{equation*}
\min_{x\in Q}L_\mathfrak S(x)=\min_{u\in Q\cap\SS^{d-1}}\min_{t\geq 0}L_\mathfrak S(tu)
\end{equation*}
immediately translates into
\begin{equation*}
K_\mathfrak S=\min_{p\in\mathcal P}K_\mathfrak S( p ).
\end{equation*}
The proof of Corollary \ref{cor:enumeration2} is completed.
\end{proof}

\section{An example of half-space walks}
\label{sec:walks_support_half-space}

In this section we illustrate the following phenomenon: If the support of the random walk is included in a certain half-space (chosen so as contradicting \ref{H2}), such a universal result as Theorem \ref{maintheorem} does not hold; in particular, the exponential decay of the non-exit probability may depend on the starting point.

Let $(S_n)_{n\geq 0}$ be the random walk on $Q$ starting at $x$ and with transition probabilities to $(1,-1)$, $(-1,1)$ and $(-1,-1)$, with respective probabilities $q$, $q$ and $p$,
%\begin{align*}
%     &\PP [S_{n+1}-S_n = (1,-1)]=\PP [S_{n+1}-S_n = (-1,1)]=q,\\
%     &\PP [S_{n+1}-S_n = (-1,-1)]=p,
%\end{align*}
where $p+2q=1$ and $p,q>0$, see Figures \ref{fig:degenerated_walks} and \ref{fig:degenerated_walks_detailed}. Let $\tau_Q$ be the exit time \eqref{eq:def_tau_K} of this random walk from the quarter-plane $Q$. Finally, define  for fixed $N$ the segment $D_{2N}=\{(i,j)\in Q : i+j=2N\}$.

\begin{proposition}
\label{prop:half-space}
For any $N\geq 1$ and any $x \in D_{2N}$, we have
\begin{equation}
\label{eq:prop:half-space}
     \lim_{n\to\infty}\PP^x[\tau_Q>n]^{1/n} = 2q\cos\left(\frac{\pi}{2N+2}\right).
\end{equation}
\end{proposition}

We shall need the following result on the simple symmetric random walk on $\ZZ$ (the proof can be easily derived from the identities in \cite[page 243]{Sp64}):
\begin{lemma}[\cite{Sp64}]
\label{lem:Sp}
For the simple symmetric random walk $(\widetilde S_n)_{n\geq 0}$ on $\mathbb Z$ {\rm(}with jumps to the left and to the right with equal probabilities $1/2${\rm)}, we have, for any $x\in\llbracket 0,2N\rrbracket$,
\begin{equation*}
     \lim_{n\to\infty}\PP^x[\widetilde S_1,\ldots,\widetilde S_n\in\llbracket 0,2N\rrbracket]^{1/n} =\cos\left(\frac{\pi}{2N+2}\right).
\end{equation*}
\end{lemma}

  \unitlength=0.6cm
\begin{figure}[t]
  \begin{center}
\begin{tabular}{ccccc}
    \hspace{-0.9cm}
        \begin{picture}(4,7.5)
    \thicklines
    \put(1,1){{\vector(1,0){7}}}
    \put(1,1){\vector(0,1){7}}
    \put(3,1){\line(-1,1){2}}
    \put(5,1){\line(-1,1){4}}
    \put(7,1){\line(-1,1){6}}
    \put(6.84,0.84){$\bullet$}
    \put(5.84,1.84){$\bullet$}
    \put(4.84,2.84){$\bullet$}
    \put(3.84,3.84){$\bullet$}
    \put(2.84,4.84){$\bullet$}
    \put(1.84,5.84){$\bullet$}
    \put(0.84,6.84){$\bullet$}
    \put(4.84,0.84){$\bullet$}
    \put(3.84,1.84){$\bullet$}
    \put(2.84,2.84){$\bullet$}
    \put(1.84,3.84){$\bullet$}
    \put(0.84,4.84){$\bullet$}
    \put(1.84,1.84){$\bullet$}
    \put(2.84,0.84){$\bullet$}
    \put(0.84,2.84){$\bullet$}
    \put(0.84,0.84){$\bullet$}
    \thinlines
    \put(4.2,4.2){\textcolor{blue}{\vector(1,-1){1}}}
    \put(4.2,4.2){\textcolor{blue}{\vector(-1,1){1}}}
    \put(4,4){\textcolor{blue}{\vector(-1,-1){0.92}}}
    \linethickness{0.1mm}
    \put(1,2){\dottedline{0.1}(0,0)(5,0)}
    \put(1,3){\dottedline{0.1}(0,0)(4,0)}
    \put(1,4){\dottedline{0.1}(0,0)(3,0)}
    \put(1,5){\dottedline{0.1}(0,0)(2,0)}
    \put(1,6){\dottedline{0.1}(0,0)(1,0)}
    \put(1,7){\dottedline{0.1}(0,0)(0,0)}
    \put(2,1){\dottedline{0.1}(0,0)(0,5)}
    \put(3,1){\dottedline{0.1}(0,0)(0,4)}
    \put(4,1){\dottedline{0.1}(0,0)(0,3)}
    \put(5,1){\dottedline{0.1}(0,0)(0,2)}
    \put(6,1){\dottedline{0.1}(0,0)(0,1)}
    \put(7,1){\dottedline{0.1}(0,0)(0,0)}
    \put(4.6,0.3){$2N$}
    \put(6.6,0.3){$2N+2$}
    \put(-0.1,4.8){$2N$}
    \put(-1.25,6.8){$2N+2$}
    \put(0.4,0.3){$0$}
    \put(4.9,3.9){\textcolor{blue}{$q$}}
    \put(3.9,4.9){\textcolor{blue}{$q$}}
    \put(4.08,4.08){\tiny \textcolor{blue}{$\bullet$}}
    \put(3.1,3.6){\textcolor{blue}{$p$}}
    \end{picture}
    \end{tabular}
  \end{center}
  \vspace{-4mm}
\caption{Random walks considered in the proof of Proposition \ref{prop:half-space} on the lines $\{(i,j)\in Q : i+j=2N\}$ for $N\geq 0$}
\label{fig:degenerated_walks_detailed}
\end{figure}
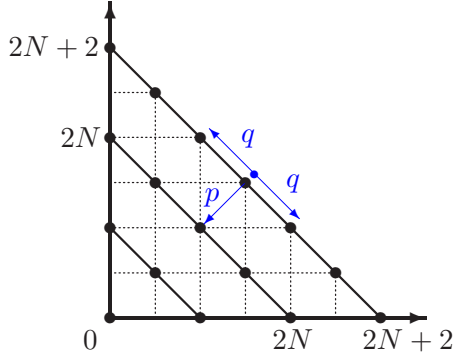

\begin{proof}[Proof of Proposition \ref{prop:half-space}]
We shall prove Proposition \ref{prop:half-space} by induction over $N\geq 1$. For $N=1$, we have three choices for $x$ (see Figure \ref{fig:degenerated_walks_detailed}). We write the proof when $x=(1,1)$, since the arguments for other values of $x$ are quite similar. For this choice of $x$, the origin $(0,0)$ can be reached only at odd times $n$, and in that event, the random walk gets out of $Q$ at time $n+1$. From this simple remark we deduce that (below, we note $(X_k)_{k\geq 1}$ the increments of the random walk $(S_n)_{n\geq 0}$)
\begin{align*}
     \PP^x[\tau_Q>2n] &= \PP^x[\tau_Q>2n, X_k\neq (-1,-1),\forall k\in\llbracket 1,n\rrbracket] \\    
     &= \PP^x[\tau_Q>2n\vert \,X_k\neq (-1,-1),\forall k\in\llbracket 1,n\rrbracket] (2q)^n.
\end{align*}
Further, the random walk conditioned on never making the jump $(-1,-1)$ is a simple symmetric random walk on the segment $D_{2}$. Therefore, 
\begin{equation*}
     \PP^x[\tau_Q>2n] = \PP^1[\widetilde S_1,\ldots ,\widetilde S_n\in\llbracket 0,2\rrbracket](2q)^{2n}.
\end{equation*}
Using Lemma \ref{lem:Sp}, we conclude that for $x=(1,1)$,
\begin{equation*}
     \lim_{n\to\infty}\PP^x[\tau_Q>2n]^{1/(2n)} = 2q\cos\left(\frac{\pi}{4}\right).
\end{equation*}
The fact that $\PP^x[\tau>n]$ is decreasing in $n$ implies that the above equation holds with $2n+1$ instead of $2n$. This achieves the proof of  Proposition \ref{prop:half-space} for $N=1$.

Let us now assume that equation \eqref{eq:prop:half-space} holds for a fixed value of $N\geq 1$. For $x\in D_{2N+2}$, introduce 
\begin{equation*}
     H :=\inf\{n>0: S_n\in D_{2N}\}
\end{equation*}
the hitting time of the set $D_{2N}$, see Figure \ref{fig:degenerated_walks_detailed}. We can write
\begin{equation}
\label{eq:decomposition}
     \mathbb P^x[\tau_Q>n] =\mathbb P^x[\tau_Q>n,H>n]+\sum_{k=1}^{n}\mathbb P^x[\tau_Q>n,H=k].
\end{equation}
The first term in the right-hand side of \eqref{eq:decomposition} can be written as
\begin{equation*}
     \mathbb P^x[\tau_Q>n,H>n] = \mathbb P^x[\tau_Q>n\vert H>n](2q)^n,
\end{equation*}
where (for the same reasons as for the case $N=1$)
\begin{equation*}
     \lim_{n\to\infty}\mathbb P^x[\tau_Q>n\vert H>n]^{1/n} =\cos\left(\frac{\pi}{2N+4}\right).
\end{equation*}
As for the second term in the right-hand side of \eqref{eq:decomposition}, 
\begin{align*}
     \mathbb P^x[\tau_Q>n,H=k] &= \EE^x[\tau_Q>k,H=k,\PP^{S_k}[\tau_Q>n-k]]\\
     &\leq C\cdot \mathbb P^x[\tau_Q>k,H=k] \mathbb P^{x_0}[\tau_Q>n-k]\\
     &\leq C\cdot {\mathbb P^x[\tau_Q>k-1,H>k-1]} \mathbb P^{x_0}[\tau_Q>n-k]\\
     &:= C a_k b_{n-k}.
\end{align*}
The first equality above comes from the strong Markov property. The first inequality follows from the fact that for any fixed $x_0\in D_{2N}$, there exists a constant $C>0$ such that, for any $n\geq 0$ and any $y\in D_{2N}$, $\PP^y[\tau_Q>n]\leq C\PP^{x_0}[\tau_Q>n]$. The second inequality is obvious, and the last line has to be read as a definition.

Using on the one hand the same reasoning as for the case $N=1$, and on the other hand the induction hypothesis, we obtain
\begin{equation*}
     \lim_{n\to\infty}a_n^{1/n} = 2q\cos\left(\frac{\pi}{2N+4}\right)>
      2q\cos\left(\frac{\pi}{2N+2}\right)=\lim_{n\to\infty}b_n^{1/n}.
\end{equation*}
Standard properties of the Cauchy product then lead to
\begin{equation*}
     \limsup_{n\to\infty}\left(\sum_{k=1}^{n}a_k b_{n-k}\right)^{1/n}\leq 2q\cos\left(\frac{\pi}{2N+4}\right).
\end{equation*}     
To summarize, with the help of \eqref{eq:decomposition} we have written $\PP^x[\tau_Q>n] = A_n+B_n$, where 
\begin{equation*}
     \lim_{n\to\infty}A_n^{1/n}=2q\cos\left(\frac{\pi}{2N+4}\right),\quad \limsup_{n\to\infty}B_n^{1/n}\leq 2q\cos\left(\frac{\pi}{2N+4}\right).
\end{equation*}     
The formula \eqref{eq:prop:half-space} therefore holds for $N+1$, and Proposition \ref{prop:half-space} is proved.
\end{proof}

\section{The case of random walks with drift in the cone}
\label{sec:appendix}

In this section we refine a result announced in~\cite{Gar07} concerning the non-exponential decay of the non-exit probability \eqref{eq:def_proba} from a cone for a centered, square-integrable and non-degenerate multidimensional random walk. We prove that the result in~\cite{Gar07} still holds if the random walk has a drift in the cone and if the hypothesis that its variance-covariance matrix is non-degenerate is weakened to \ref{H1}. 
This result is one of the main ingredients of the proof of Theorem \ref{maintheorem}. We would like to notice that the proof of (a weakened form of) Proposition~\ref{theorem_cras_extended} below was only sketched in \cite{Gar07}, so that the extended proof we shall give here is new.

As before, we only assume that $K$ is a closed convex cone with non-empty interior $\open{K}$. We fix some $v\in \open{K}$ and define $K_{\delta}=K+\delta v$. In this setting, we shall prove the following:

\begin{proposition}
\label{theorem_cras_extended} 
Assume that the distribution $\mu$ of the random walk increments is square-integrable and truly $d$-dimensional \ref{H1}. Suppose in addition that
the drift $m=\EE^0[S_1]$ belongs to the cone $K$ and that $v$ is a vector orthogonal to $m$. Then there exists $\alpha>0$ and $\delta\geq 0$ such that, for all $x\in K_{\delta}$,
\begin{equation*}
\lim_{n\to\infty}\PP^x[\tau_K>n, \vert\sclr{v}{S_n}\vert\leq \alpha\sqrt{n}]^{1/n}=1.
\end{equation*}
\end{proposition}

If $m=0$, as we have already explained, \ref{H1} is equivalent to the fact that the variance-covariance matrix of the increments distribution is non-degenerate. Hence, we exactly recover~\cite[Theorem]{Gar07}. However, if $m\neq 0$, Proposition~\ref{theorem_cras_extended} can not be derived from \cite[Theorem]{Gar07}. Indeed, under the hypothesis of Theorem~\ref{theorem_cras_extended}, it is clear that
\begin{equation*}
\PP^x[\tau_K>n, \vert\sclr{v}{S_n}\vert\leq \alpha\sqrt{n}]\geq \PP^x[\tau_K(\widetilde{S})>n, \vert\sclr{v}{\widetilde{S}_n}\vert\leq \alpha\sqrt{n}],
\end{equation*}
where $(\widetilde{S}_n=S_n-nm)_{n\geq 0}$ is the centered random walk associated with $(S_n)_{n\geq 0}$. But the variance-covariance matrix of $\widetilde{S}_1$ is equal to that of $S_1$ and might be degenerate, so that~\cite[Theorem]{Gar07} would not apply to the walk $(\widetilde{S}_n)_{n\geq 0}$. This is for example the case when $(S_n)_{n\geq 0}$ is the uniform two-dimensional random walk with step set $\mathfrak S=\{(0,1),(1,0)\}$.

In order to prove Theorem~\ref{theorem_cras_extended}, we will need a series of lemmas. We begin with some geometric considerations.

\begin{lemma}
\label{intersection_cone_subspace}
Let $m$ be a point in $K$ and $V$ be a linear subspace of $\RR^d$. If $(m+V)\cap\open{K}=\emptyset$, then $m+V$ is included in a hyperplane.
\end{lemma}

\begin{proof} 
Clearly, $\dim V<d$, for else $m+V=\RR^d$ would intersect $\open{K}$. Since $m+V$ is included in the linear subspace %$[m,V]$ 
generated by $m$ and $V$, whose dimension is less or equal to $\dim V+1$, the  affine space $m+V$ is always included in a hyperplane if $\dim V<d-1$. Thus, it remains to consider the case where $V$ is a hyperplane, that is, 
\begin{equation*}
     V=u^\perp=\{x\in\RR^d: \sclr{u}{x}=0\}
\end{equation*}
for some $u\not=0$. Assume $m\notin V$. Possibly changing $u$ to $-u$, we can assume in addition that $\sclr{m}{u}>0$. Then,
\begin{equation*}
     \bigcup_{\lambda>0}(\lambda m +V)=\bigcup_{\lambda >0}(\lambda u+V)=\{x\in\RR^d : \sclr{u}{x}>0\}=:u^+_*.
\end{equation*}
Further, by homogeneity of $V$ and $\open{K}$, we have $(\lambda m+V)\cap \open{K}=\emptyset$ for every $\lambda>0$. Hence, $\open{K}$ does not intersect $u^+_*$, and is therefore included in $u^-=\{x\in\RR^d : \sclr{u}{x}\leq 0\}$. This is a contradiction since $m\in K=\close{(\open{K})}$ (this equality holds for any convex set with non-empty interior) and $\sclr{m}{u}>0$. Thus $m$ belongs to $V$ and  $m+V=V$ is a hyperplane.
\end{proof}

\begin{lemma}
\label{gaussian_probability_of_the_cone}
Let $Y\in\RR^d$ be a random vector with Gaussian distribution $\N(m,\Gamma)$. If $m$ belongs to $K$ and if $m+(\ker\Gamma)^{\perp}$ is not included in a hyperplane, then 
\begin{equation*}
     \PP[Y\in\open{K}]>0.
\end{equation*}
\end{lemma}

\begin{proof}
It is well known that the Gaussian distribution $\N(m,\Gamma)$ admits a positive density with respect to Lebesgue measure on the affine space $m+(\ker\Gamma)^\perp$. Thus, it suffices to show that $(m+(\ker\Gamma)^\perp)\cap \open{K}$ is a non-empty open set in $m+(\ker\Gamma)^\perp$. But this follows from Lemma~\ref{intersection_cone_subspace} since $m+(\ker\Gamma)^\perp$ is not contained in a hyperplane.
\end{proof}

\begin{lemma}
\label{path_to_push_the_walk}
Under the hypotheses of Theorem \ref{theorem_cras_extended}, there exist $k\geq 1$ and $\delta\geq 0$ such that
\begin{equation*}
\PP[\tau_{K_{-\delta}}>k, S_k\in\open{K}]>0.
\end{equation*}
\end{lemma}

\begin{proof} 
For $n\geq 1$ define $Z_n:=(S_n-nm)/\sqrt{n}$. The homogeneity and convexity properties of $K$ ensure that
\begin{equation*}
     \{S_n\in\open{K}\}=\{Z_n\in\open{K}-\sqrt{n}m\}\supset \{Z_n\in \open{K}-m\}
\end{equation*}
Let $\Gamma$ denote the variance-covariance matrix of $\mu$. We notice that $m+(\ker\Gamma)^{\perp}$ is not included in a hyperplane since $\mu$ satisfies~\ref{H1}. The central limit theorem asserts that $(Z_n)_{n\geq 1}$ converges in distribution to a random vector $Y$ with Gaussian distribution $\N(0,\Gamma)$. Hence, applying the Portmanteau theorem~\cite[Theorem 2.1]{Bil68}, we obtain the inequality
\begin{equation*}
\liminf_{n\to\infty}\PP[S_n\in\open{K}]\geq\PP[Y\in\open{K}-m],
\end{equation*}
where the right-hand side is positive according to Lemma~\ref{gaussian_probability_of_the_cone}. Now, to conclude it suffices to fix $k$ so that $\PP[S_k\in\open{K}]=2\epsilon>0$, and then choose $\delta$ so large that $\PP[\tau_{K_{-\delta}}>k]\geq 1-\epsilon$ (this is possible since $K_{-\delta}\uparrow\RR^d$ as $\delta\uparrow\infty$).
\end{proof}

\begin{proof}[Proof of Theorem~\ref{theorem_cras_extended}]
The proof follows the same kind of arguments as in~\cite{Gar07}. We shall first use Lemma~\ref{path_to_push_the_walk} in order to push the random walk inside the cone  at a distance $\sqrt{n}$ from the boundary, and then apply the functional central limit theorem.

Recall that $v\in\open{K}$ is fixed and that $K_{\delta}=K+\delta v$. We know by Lemma~\ref{path_to_push_the_walk} that there exist $\delta\geq 0$ and $k\geq 1$ such that
\begin{equation*}
     \PP[\tau_{K_{-\delta}}>k, S_{k}\in \open{K}]>0.
\end{equation*}
Therefore, we can find a closed ball $B:=\close{B(z,\epsilon)}\subset \open{K}$, with center at $z\in\open{K}$ and radius $\epsilon>0$, such that 
\begin{equation*}
     \PP[\tau_{K_{-\delta}}>k, S_{k}\in B]=\gamma>0.
\end{equation*}
Since $K$ is a convex cone, it satisfies the relation $K+K\subset K$, thus
\begin{equation*}
     \PP^x[\tau_{K}>k, S_{k}-x\in B]\geq \gamma,
\end{equation*}
for all $x\in K_{\delta}$ (by inclusion of events). From this, we shall deduce by induction that
\begin{equation}
\label{pousselamarche2}
p_\ell:=\PP^x[\tau_K>\ell k,S_{\ell k}-x\in \ell B]\geq \gamma^\ell,
\end{equation}
for all $\ell\geq 1$ and $x\in K_{\delta}$. Indeed, by the Markov property of the random walk,
\begin{align*}
p_{\ell+1} & \geq \EE^x[\tau_K>\ell k,S_{\ell k}-x\in \ell B, \PP^{S_{\ell k}}[\tau_K>k,S_{k}-x\in (\ell+1)B]]\\
& \geq p_\ell\cdot\inf_{\{y\in K: y-x\in\ell B\}} \PP^{y}[\tau_K>k,S_{k}-x\in (\ell+1)B].
\end{align*}
Noticing that $y-x\in \ell B$ and $S_k-y\in B$ yields $S_k-x\in (\ell+1)B$, we obtain
\begin{equation*}
      p_{\ell+1}\geq p_\ell\cdot\inf_{\{y\in K: y-x\in\ell B\}} \PP^{y}[\tau_K>k,S_{k}-y\in B].
\end{equation*}
But $x\in K_{\delta}$ and $y-x\in \ell B\subset K$ imply $y\in K_{\delta}$. Hence $p_{\ell+1}\geq p_\ell\cdot\gamma$ and~\eqref{pousselamarche2} is proved.

Now, for $x\in K_{\delta}$, define
\begin{equation*}
\widetilde p_n:=\PP^x[\tau_K>n, \vert\sclr{v}{S_n}\vert\leq \alpha\sqrt{n}],
\end{equation*}
where $\alpha>0$ will be fixed latter.
Write $\ell=\lfloor\sqrt{n}\rfloor$ for the lower integer part of $\sqrt{n}$. Using the Markov property at time $\ell k$ and the estimate in~\eqref{pousselamarche2} leads to
\begin{align*}
\widetilde p_n &\geq \PP^x[\tau_K>n,S_{\ell k}-x\in\ell B, \vert\sclr{v}{S_n}\vert\leq \alpha\sqrt{n}]\\
    &\geq \EE^x[\tau_K>\ell k,S_{\ell k}-x\in\ell B, \PP^{S_{\ell k}}[\tau_K>n-\ell k,\vert\sclr{v}{S_{n-\ell k}}\vert\leq \alpha\sqrt{n}]]\\
		&\geq \gamma^\ell\cdot\inf_{\{y\in K: y-x\in\ell B\}}\PP^y[\tau_K>n-\ell k, \vert\sclr{v}{S_{n-\ell k}}\vert\leq \alpha\sqrt{n}].
\end{align*}
Therefore, Proposition~\ref{theorem_cras_extended} will follow from the fact that
\begin{equation*}
\label{limit_pushed_probab}
     \liminf_{n\to\infty}\inf_{\{y\in K: y-x\in\ell B\}}\PP^y[\tau_K>n-\ell k, \vert\sclr{v}{S_{n-\ell k}}\vert\leq \alpha\sqrt{n}]>0,
\end{equation*}
which we shall prove now. Since $\ell k \ll n$ does not play any significant role in the last probability, we will neglect it in order to simplify notations.
Also, for any $\epsilon'>\epsilon$, we have 
\begin{equation*}
     x+\ell B=\ell\left(\frac{x}{\ell}+\close{B(z,\epsilon)}\right)\subset\ell \close{B(z,\epsilon')}
\end{equation*}
for all large enough $n$. Since an $\epsilon'>\epsilon$ can be found so that $\close{B(z,\epsilon')}\subset\open{K}$, we may replace $x+\ell B$ by $\ell B$ without loss of generality. Finally, since $\ell\leq\sqrt{n}$, we may replace $\ell B$ by $\sqrt{n}B$. With these simplifications, it remains to consider
\begin{equation*}
q_n:=\inf_{\{y\in K: y\in \sqrt{n} B\}}\PP^y[\tau_K>n, \vert\sclr{v}{S_n}\vert\leq \alpha\sqrt{n}].
\end{equation*}
By mapping $y$ to $y/\sqrt{n}$, we may write
\begin{equation*}
     q_n=\inf_{y\in B}q_n(y),
\end{equation*}
where
\begin{equation*}
q_n(y):=\PP^0[\tau_K(y\sqrt{n}+S)>n, \vert\sclr{v}{y\sqrt{n}+S_n}\vert\leq \alpha\sqrt{n}].
\end{equation*}
Let $\widetilde{S}=(\widetilde{S}_n=S_n-nm)_{n\geq0}$ denote the centered random walk associated with $S=(S_n)_{n\geq 0}$. By inclusion of events we get the lower bound
\begin{equation*}
     q_n(y)\geq \PP^0[\tau_K(y\sqrt{n}+\widetilde{S})>n, \vert\sclr{v}{y\sqrt{n}+\widetilde{S}_n}\vert\leq \alpha\sqrt{n}],
\end{equation*}
where we used the fact that $m\in K$ and $\sclr{v}{m}=0$.
Finally, let us denote by $Z_n=(Z_n(t))_{t}$ the random process with continuous paths that coincides with $\widetilde{S}_k/\sqrt{n}$ for $t=k/n$ and which is linearly interpolated elsewhere. By definition of $Z_n$ and convexity of $K$, the last inequality immediately rewrites
\begin{equation*}
     q_n(y)\geq \PP^0[\tau_K(y+Z_n(t))>1, \vert\sclr{v}{y+Z_n(1)}\vert\leq \alpha].
\end{equation*}
The functional central limit theorem ensures that $Z_n$ converges in distribution to a Brownian motion $(b(t))_t$ with variance-covariance matrix $\Gamma$. Suppose that the sequence $(y_n)_{n\geq0}$ converges to some $y\in B$. Then $(y_n+Z_n)_{n\geq 0}$ converges in distribution to $y+b$, and it follows from the Portmanteau theorem that
\begin{equation}
\label{last_label}
     \liminf_{n\to\infty}q_n(y_n)\geq \PP^0[\tau_{\open{K}}(y+b(t))>1, \vert\sclr{v}{y+b(1)}\vert< \alpha].
\end{equation}
Now, it is time to choose $\alpha$. To do this, first recall that $B=\close{B(z,\epsilon)}\subset\open{K}$. Choose $\eta>\epsilon$ so that
$\close{B(z,\eta)}\subset\open{K}$ and set (notice that we could have done this at the very beginning of the proof)
\begin{equation*}
     \alpha=\vert v\vert(\vert z\vert+\eta).
\end{equation*}     
If $\vert b(t)\vert <\eta-\epsilon$ for all $t\in [0,1]$, then $y+b(t)\in \close{B(z,\eta)}\subset\open{K}$ for all $t\in[0,1]$. Furthermore,
\begin{equation*}
     \vert\sclr{v}{y+b(1)}\vert\leq\vert v\vert(\vert y\vert+\vert b(1)\vert)<\vert v\vert(\vert z\vert+\epsilon+\eta-\epsilon)=\alpha.
\end{equation*}     
Therefore, the probability in \eqref{last_label} is bounded from below by the probability
\begin{equation*}
     \PP\left[\max_{t\in[0,1]}\vert b(t)\vert <\eta-\epsilon\right]
\end{equation*}     
that the Brownian motion $(b(t))_t$ stays near the origin for all $t\in[0,1]$, and this event happens with positive probability, regardless $\Gamma$ be positive definite or not.

To summarize, we have proved that
\begin{equation*}
\liminf_{n\to\infty}q_n(y_n)>0,
\end{equation*}
for any sequence $(y_n)_{n\geq0}\in B$ that converges to some $y$. Thus, by standard compactness arguments, we reach the conclusion that
\begin{equation*}
\liminf_{n\to\infty}q_n=\liminf_{n\to\infty}\inf_{y\in B}q_n(y)>0,
\end{equation*}
and Proposition~\ref{theorem_cras_extended} is proved.
\end{proof}

\section*{Acknowledgments}
We thank Marni Mishna for motivating discussions. Many thanks also to Marc Peign\'e for his valuable comments and encouragement. We thank two anonymous referees for useful comments and suggestions.

\end{document}